\documentclass[pdflatex,sn-mathphys-num]{sn-jnl}
\usepackage{mathtools}

\usepackage{graphicx}%
\usepackage{multirow}%
\usepackage{amsmath,amssymb,amsfonts}%
\usepackage{amsthm}%
\usepackage{mathrsfs}%
\usepackage[title]{appendix}%
\usepackage{xcolor}%
\usepackage{textcomp}%
\usepackage{manyfoot}%
\usepackage{booktabs}%
\usepackage{algorithm}%
\usepackage{algorithmicx}%
\usepackage{algpseudocode}%
\usepackage{listings}%

\theoremstyle{thmstyleone}%
\newtheorem{theorem}{Theorem}
\newtheorem{coro}[theorem]{Corollary}%

\theoremstyle{thmstyletwo}%

\theoremstyle{thmstylethree}%
\newtheorem{definition}{Definition}%

\theoremstyle{thmstylethree}%
\newtheorem{conjecture}{Conjecture}%

\theoremstyle{thmstylethree}%
\newtheorem{lemma}{Lemma}%

\newcommand{\C}{{\mathbb{C}}}
\newcommand{\Z}{{\mathbb{Z}}}
\newcommand{\Q}{{\mathbb{Q}}}
\newcommand{\N}{{\mathbb{N}}}

\newcommand{\Gal}{{\mathrm{Gal}}}

\newcommand{\p}{{\mathfrak{p}}}

\newcommand{\Nm}{{\mathrm{Nm}}}

\renewcommand{\O}{{\mathcal{O}}}

\begin{document}

\title[The $abc$ conjecture implies infinitely many non-Wieferich places for fixed bases in number fields]{The $abc$ conjecture implies infinitely many non-Wieferich places for fixed bases in number fields}


\author*[1]{\fnm{Hester} \sur{Graves}}\email{hkgrave@super.org}

\author[2]{\fnm{Benjamin} \sur{Weiss}}\email{Benjamin.Leonard.Weiss@gmail.com}


\affil*[1]{
\orgname{IDA/CCS}, 
\city{Bowie}, \postcode{20715}, \state{Maryland}, \country{USA}}

\affil[2]{\orgdiv{Unum: Attn Benjamin Weiss}, \orgname{Mailstop B299}, \orgaddress{\street{2211 Congress Street}, \city{Portland}, \postcode{04122}, \state{Maine}, \country{USA}}}




\baselineskip=17pt




\abstract{Silverman showed that, assuming the $abc$ conjecture, there are $\gg \log x$ non-Wieferich primes base $a$
 less than $x$ \cite{silverman}, for all non-zero $a$.  
This inspired Graves and Murty \cite{Graves}, Chen and Ding \cite{Chen1} \cite{Chen2}, and then Ding \cite{Ding} to find growth results, 
assuming the $abc$ conjecture,
 for non-Wieferich primes $p$ base $a$, where $p \equiv 1 \pmod{k}$ for integers $k \geq 2$.  
 In light of Murty, Srinivas, and Subramani's recent work on `the Wieferich primes conjecture'
and Euclidean algorithms in number fields \cite{murty}, number theorists need results on non-Wieferich places in number fields.
We prove analogues of the results of Graves \& Murty and Ding, and show 
Ding's result holds for all bases $a$ in all imaginary quadratic fields' rings of integers, with $31$ explicitly listed exceptions.

  Along the way, we generalize useful results on rational integers to algebraic integers.}




\maketitle




\section{Introduction}

For all primes $p$ and all integers $a$ not divisible by $p$, Fermat's little theorem states $a^{p-1} \equiv 1 \pmod{p}$.  
If $a^{p-1} \equiv 1 \pmod{p^2}$, $p$ is a \textbf{Wieferich prime base $a$}; 
otherwise, it is a \textbf{non-Wieferich prime base $a$}.  
Wieferich introduced this classification pursuing Fermat's Last Theorem, but it is still useful
today; for example, Mil\u{a}ilescu \cite{mihailescu} proved Catalan's conjecture using Wieferich primes. 
In spite of their utility, the only known Wieferich primes base two are $1093$ and $3511$ \cite{oeis}. The two known Wieferich primes 
base three are $11$ and $1006003$, and there are no known Wieferich primes base 47 \cite{oeis}.

M. Ram Murty and his school use a generalization of non-Wieferich primes, `admissible sets of primes,' to study Euclidean 
domains.  If $K$ is a quadratic number field, and if $\varepsilon$ is a fundamental unit of $K$, an admissible prime is a 
non-Wieferich place base $\varepsilon.$  Murty, Srinivas, and Subramani recently showed that if one assumes the Hardy-Littlewood 
conjecture and the `Wieferich primes conjecture,' a growth result on sets of non-Wieferich places with multiplicative units of
infinite order as their base, then all real quadratic number fields with class number one are Euclidean \cite{murty}.
This new approach encourages studying growth results of non-Wieferich places in number fields, especially 
real quadratic number fields.  Assuming the $abc$ conjecture, Theorem \ref{iqf} gives a growth result for all bases (except for
thirty-one explicitly stated exceptions) for all imaginary quadratic fields.

There are few unconditional results on Wieferich primes. 
It is unknown whether there are infinitely many non-Wieferich primes 
for some non-unit base.  In his seminal paper ``Wieferich's criterion and the ABC conjecture,'' Silverman \cite{silverman} showed that if the 
$abc$ conjecture were true, then there would be $ \gg \log x$ 
non-Wieferich primes less than $x$ for all bases $a$.  
Since then, DeKoninck and Doyon \cite{dekoninck} proved the same growth result under the weaker assumption that for some $\epsilon >0$, the set 
$\{n \in \N: \frac{\log (2^n -1)}{\log (\text{rad}(2^n -1))} < 2- \epsilon \} $ has density one.   
Graves and Murty \cite{Graves} showed that, assuming the $abc$ conjecture,
for all non-units $a$ and all integers $k \geq 2$,  there are 
$ \gg \log x /\log \log x$ non-Wieferich primes $p$ base $a$ with $p \leq x$ and $p \equiv 1 \pmod{k}$.
Chen and Ding improved the growth result to $\gg (\log x / \log \log x) (\log \log \log x)^M$ for any fixed $M$ \cite{Chen1} \cite{Chen2},
which Ding then substantially strengthened to $\gg \log x $ \cite{Ding}.


This article generalizes these growth results to Galois number fields and Wieferich places $\p$ base $a$ via the following theorems.

\begin{theorem}\label{basic} Suppose that $K$ is a Galois number field,  that $\varepsilon >0$, that $k \in \Z^+$ , and that
$a$ is a non-zero algebraic integer (i.e. $a \in \O_K \setminus \{0\}$) that is not a root of unity.
If one assumes the abc conjecture for $K$, then there are $\gg_{\varepsilon, K, a,k} \log x/ \log \log x$ non-Wieferich places
$\p$ base $a$ such that $\Nm(\p) \leq x$ and $\Nm(\p) \equiv 1 \pmod{k}$.
\end{theorem}

\begin{theorem}\label{strong} Suppose $K$ is a Galois number field,  $k \in \Z^+$ is a positive integer, and that
$a \in \O_K$ satisfies $| \sigma(a)| \geq 2$ for all $\sigma \in \Gal(K/\Q)$.
Assuming the abc conjecture for $K$, then there are $\gg_{K, a, k} \log x$ non-Wieferich places
$\p$ base $a$ such that $\Nm(\p) \leq x$ and $\Nm(\p) \equiv 1 \pmod{k}$.
\end{theorem}

\begin{theorem}\label{iqf} If $K$ is an imaginary quadratic number field, if $k \in \Z^+$, 
if the $abc$ conjecture holds in $K$, and if $a \in \O_K$ but not in the set 
\[
\left \{
\begin{array}{l}
0, \pm 1, \pm i, \pm 1 \pm i, \pm \sqrt{-2}, \pm 1 \pm \sqrt{-2}, \\
\frac{\pm 1 \pm \sqrt{-3}}{2},  \frac{\pm 3 \pm \sqrt{-3}}{2},
\frac{\pm 1 \pm \sqrt{-7}}{2},\frac{\pm 1 \pm \sqrt{-11}}{2} 
\end{array}\right \},\] 
then there are $\gg_{ K,a, k} \log x$ non-Wieferich places $\p$ base $a$ with $\Nm (\p) \leq x$ and $\Nm (\p)\equiv 1 \pmod{k}$.
\end{theorem}
 
 The proofs use new results on algebraic integers.  While they are generalizations of known 
 results on rational integers, we have not seen our statements in the literature.  We hope 
  readers find them useful in their own research.  

\section{Background}
\subsection{The $abc$ conjecture and the $abc$ conjecture for number fields}

Given an integer $n =  \prod_{i=1}^k p_i^{a_i}$, where the $p_i>0$ are distinct primes and the $a_i$ are positive, the \textbf{radical} of $n$, or $\text{rad}(n)$, is 
$\prod_{i=1}^k p_i$. 

\begin{conjecture}($abc$ Conjecture of Oesterl\'{e} and Masser, 1985 \cite{oesterle}, \cite{masser}) If
 $(a,b,c) \in \Z^3$, with $a+b =c$ and $(a,b)=1$, then for all $\epsilon >0$,
\[c 
\ll_{\epsilon} 
\left (\text{rad}(abc) \right )^{1 + \epsilon}.\]\end{conjecture}

Conjecture \ref{nf_abc_conj} below is equivalent to Elkies' version (\cite{elkies}, p 100) of Vojta's \textbf{$abc$ conjecture for number fields}, 
which originally appeared on page 84 of \cite{vojta}.
Before stating the generalized $abc$ conjecture, we introduce the following definitions and notations,
including $(\alpha)$ to represent the ideal generated by $\alpha$.

For $a_1, a_2, a_3 \in K^{\times}$, define the \textbf{height} and \textbf{conductor} \cite{granville} respectively by

\begin{align*}
H(a_1,a_2,a_3)& := \prod_v \max \left (\|a_1 \|_v, \|a_2\|_v, \|a_3\|_v \right )\\
N(a_1,a_2,a_3) & := \prod_{\p \in I} \Nm_{K/\Q}(\p)^{\frac{1}{[K:\Q]}},
\end{align*}
where the first product ranges over all infinite places of $K$ with normalization $\|a\|_{v} = |v(a)|^{\frac{1}{[K:\Q]}},$
 and $I$ is the (finite) set of normalized finite places $\p$ such that 
$\|a_1\|_{\p}$, $\|a_2\|_{\p}$, $\|a_3\|_{\p}$ are not all equal to $1$.

\begin{conjecture}\label{nf_abc_conj}(The $abc$ Conjecture for a  Galois Number Field $K$) Let  $\alpha,\beta \in K$, where $K$ is a Galois number field. 
For every $\epsilon >0$, 
\[H(\alpha,\beta, \alpha +\beta) \ll_{\epsilon, K} N\left ( (\alpha,\beta, \alpha +\beta) \right )^{1 + \epsilon}.\]
\end{conjecture}
Note that the conjecture implies that if $\alpha + \beta$ is a root of unity, then 
\begin{equation}\label{Gal_implication}
\max ( |\Nm(\alpha)|, |\Nm(\beta)|) \ll_{\epsilon, K} 
\left (  \prod_{\p | (\alpha)} \Nm(\p)   \prod_{\p | (\beta)} \Nm(\p)   \right )^{1 + \epsilon}.
\end{equation}

\subsection{Wieferich places}
 Given a rational prime $p$, if $p^a|b$ but $p^{a+1} \nmid b$, we write $v_p(b) =a$.
If $\p$ is a prime ideal and if $I$ and $J$ are ideals with $I = \p^a \cdot J$ and $\p \nmid J$, we say $v_{\p}(I) =a$.  
We call $v_{\p}$ the $\p$-adic valuation of the ideal I.
Given an element $x$, $v_{\p}(x)$ is the $\p$-adic valuation of the ideal $(x)$.
When $\p \nmid I$, $v_{\p}(I) =0$. 
  
\begin{definition}\label{factorizations}Let $K$ be a Galois number field with ring of integers $\O_K$.  
Suppose $a \in \O_K \setminus \{0\}$ and $\p$ is a prime ideal with $|\Nm_{K/\Q}(\p)|=q$.  
If $v_{\p}(a^{q-1}-1)$ is greater than one,  $\p$ is a \textbf{Wieferich place base $a$}.  
Otherwise, $v_{\p} (a^{q-1} -1)=1$ and $\p$ is a \textbf{non-Wieferich place base $a$}.\end{definition}

Given an ideal $I$ in $K$ with $I = \prod \p_i^{\alpha_i}$,  the \textbf{powerful part} of $I$ is 
$\prod_{\alpha_i \geq 2} \p_i^{\alpha_i}$.  We define the \textbf{radical} of $I$ by $\text{rad}(I) =\prod \p_i$.
Recall that, given ideals $I$ and $J$, their sum $I + J$ is the smallest ideal that contains both of them.
This notation can be counter-intuitive as $(I+J)|I$, no matter the choice of $J$. 
We use $\Phi_n(x)$ to denote the $n$th cyclotomic polynomial.
This notation allows us to restate equation \ref{Gal_implication}	as 
\[ \max (|\Nm(\alpha)|, |\Nm(\beta)| \ll_{\epsilon, K} \left ( \Nm(\text{rad}(\alpha)) \Nm( \text{rad}(\beta)) \right )^{1 + \epsilon}.\]

\begin{definition}  Given some $a \in \O_K \setminus \{0\}$ that is not a root of unity, we factor the ideal
$(a^n -1)$ into $C_{n,a} D_{n,a}$, where $D_{n,a}$ is the powerful part of $(a^n -1)$.
We define the ideals $C'_{n,a} = (\Phi_n(a)) + C_{n,a}$ and $D_{n,a}' = (\Phi_{n}(a)) + D_{n,a}$.  
Since  $C_{n,a}$ is a square-free ideal, so is $C_{n,a}'$. 
\end{definition}

The ideals $C_{n,a}$ and $D_{n,a}$ are relatively prime, so 
\[C_{n,a} + D_{n,a} = C_{n,a}' + D_{n,a}' = \O_K.\]
In order to prove Theorems \ref{basic}, \ref{strong}, and \ref{iqf}, we need the following lemmas, 
respectively adapted from Silverman \cite{silverman} and Chen and Ding \cite{Chen1} \cite{Chen2}.
If $\p$ divides $a^n$ for some $n >0$, then we denote the smallest such $n$ by $e_{\p}(a)$ and call it the order of $a$ modulo $\p$.

\begin{lemma}\label{heart} (adapted from Lemma 3, \cite{silverman}, p 229) If $a \in \O_K \setminus \{0\}$ is not a root of unity and 
 if $\p$ is a prime ideal of $\O_K$ dividing $C_{n,a}$, then $\p$ is a non-Wieferich place base $a$.
\end{lemma}
\begin{proof} Suppose $\Nm(\p) = q$ and $\p|C_{n,a}$, so $a^n \equiv 1 \pmod{\p}$.  
Since $\O_K$ is Dedekind, $\p$ has at most two generators.  
We can therefore write $\p = (u,v)$ for some $u, v \in \O_K$; if $\p$ is principal, set $v=0$.  
There exist some $s, t \in \O_K$ such that $a^{e_{\p}(a)} =1 +us +vt$.
We know $(us +vt)^j \in \p^2$ for all $j>1$, so 
$$a^n = (a^{e_{\p}(a)})^{\frac{n}{e_{\p}(a)}}= (1 + us+vt)^{\frac{n}{e_{\p}(a)}}\equiv 1 + \frac{n}{e_{\p}(a)} (us+vt) \pmod{\p^2}.$$
Since $\p |C_{n,a}$,  $\p^2 $ divides neither $(a^n -1)$ nor its divisor $(a^{e_{\p}} -1)$.
We deduce $\frac{n}{e_{\p}(a)} (us+vt) \not \equiv 0 \pmod{\p^2}$, 
so $\frac{n}{e_{\p}(a)} \notin \p$ and $(us + vt) \notin \p^2$.  
Therefore
$$(a^{e_\p(a)})^{\frac{q-1}{e_\p(a)}} 
= (1 + us +vt)^{\frac{q-1}{e_{\p}(a)}} 
\equiv 1 + \frac{q-1}{e_{\p}(a)} (us+vt) \pmod{\p^2}.$$
Because $q-1$ is co-prime to $\p$, the quotient$\frac{q-1}{e_{\p}(a)} \notin \p$.
We saw that $(us+vt) \notin \p^2$, 
so we conclude $a^{q-1} \not \equiv 1 \pmod{\p^2}$ and thus $v_{\p}(a^{q-1}-1) =1$. 
\end{proof}

\begin{lemma}  \label{Chen_lemma} (adapted from \cite{Chen1}, \cite{Chen2}, p 1834-5) If $K$ is a number field,  if $a\in \O_K \setminus \{0\}$ is not a root of unity, 
and if $1 \leq m < n$, then $C'_{m,a} + C'_{n,a} =\O_K$.
\end{lemma}
\begin{proof} If a prime ideal $\p$ divides $ (C_{m,a}' + C_{n,a}')$, it divides both $C_{m,a}'$ and $C_{n,a}'$, and therefore divides
$(\Phi_m(a))$ and $(\Phi_n(a))$.  
Hence it also divides their multiples $(a^m -1)$ and $(a^n -1)$.  
The greatest common divisor of the polynomials $x^m -1$ and $x^n -1$ is $x^{\gcd(m,n)} -1$, so $\p | (a^{\gcd(m,n)} -1)$.  
We can factor the ideal $(a^n -1)$ as
\[(a^n -1) = (a^{\gcd(m,n)} -1) (\Phi_n(a)) \prod_{\tiny{\begin{array}{c} d|n, d\neq n \\ d \nmid \gcd(m,n) \end{array}}} (\Phi_d(a)).\]
Since $\p$ divides both $ (a^{\gcd(m,n)} -1)$ and $(\Phi_n(a))$, we deduce $\p^2 | (a^n -1)$.
This means $\p $ divides the powerful part of $(a^n -1)$, the ideal $D_{n,a}$,
contradicting our assumption that $\p | C_{n,a}'$ (and thus $C_{n,a}$), as $C_{n,a} + D_{n,a} = \O_K$.  
We conclude that $C_{m,a}' + C_{n,a}' =\O_K$.
  \end{proof}
  
\subsection{Useful Results on Norms of Algebraic Integers}

Here, we generalize known results on integers to number fields.  
Some of the new statements are not necessarily what one would expect. 

\begin{lemma}\label{p_equiv_1} If $K$ is a Galois number field, if $a \in \O_K$, if $\p$ is a prime ideal of $\O_K$, and  if $n>1$ is the least positive integer such that 
$\p | (\Phi_n(a))$, then $\Nm(\p) \equiv 1 \pmod{n}$.
\end{lemma}
\begin{proof}  As $\p | (\Phi_n(a))$ and $\p \nmid (\Phi_m(a))$ for $m < a$, it follows that $\p |(a^n -1)$ and $\p \nmid (a^m -1)$.  Thus $n$ is the order of $a$
modulo $\p$ and $n|(\Nm(\p) -1)$.  
\end{proof}


We can embed any Galois number field $K$ into the complex numbers $\C$ via an injection $i: K \xhookrightarrow{} \C$.  
We use $|a|$, where $a$ is neither a rational prime nor a root of unity, to denote the magnitude of the complex number $i(a)$.
Since we always (eventually) consider the entire set 
 of $|\sigma(a)|$, for $\sigma \in \Gal(K/\Q)$,
  our choice of injection does not matter, and we safely use the notation $|a|$.

\subsubsection{The norm of $a^n -1$ and its factors}\label{sub_norm_factors}

\begin{lemma}\label{upper_norm_bound} 
If $K$ is a Galois number field and $a$ is an element of $K$ such that $|\sigma(a)| \geq 1$ for all $\sigma \in \Gal(K/\Q)$, then
\[ \max ( |\Nm(a^n -1)|, |\Nm(a^n +1)| ) \leq 2^{[K:\Q]} |\Nm(a)|^n.\]
\end{lemma}
\begin{proof}  We follow Nick Ramsey's more general proof of Lemma 2.3.9 in \cite{thesis}.  
By definition, 
\[|\Nm(a^n -1)| = \left | \prod_{\sigma \in \Gal(K/\Q)} \sigma(a^n -1) \right | = \left | \prod_{\sigma \in \Gal(K/\Q)} (\sigma(a)^n -1) \right |.\]
We can rewrite this as a sum over all subsets $S \subset \Gal(K/ \Q)$, 
\begin{equation}\label{norms} \left | \sum_{S \subset \Gal(K/\Q)} (-1)^{|\Gal(K/\Q) \setminus S|} \prod_{\sigma \in S} \sigma (a^n) \right |
\leq \sum_{S \subset \Gal(K/\Q)}  \prod_{\sigma \in S} \left | \sigma(a^n)\right |.
\end{equation}
The right hand side is also an upper bound for $|\Nm(a^n +1)| = \left | \sum_{S \subset \Gal(K/\Q)}  \prod_{\sigma \in S} \sigma (a^n) \right |$.  
As $|\sigma(a)| \geq 1$ for all $\sigma \in \Gal(K/\Q)$, 
\[ \prod_{\sigma \in S} \left | \sigma(a^n)\right | \leq \prod_{\sigma \in \Gal(K/\Q)} \left | \sigma(a^n)\right | = |\Nm(a)|,\]
so we bound the right side of equation \ref{norms} above  by $\sum_{S \subset \Gal(K/\Q)} |\Nm(a)| = 2^{[K:\Q]} |\Nm(a)|$.  
\end{proof}

The next three bounds are analogues of Lemma 2.4 in \cite{Graves}.

\begin{lemma}\label{first_bound_on_D}  If $K$ is a Galois number field, if $\varepsilon >0$, if $a \in \O_K \setminus \{0\}$ is not a root of unity, and if the $abc$ conjecture holds in $K$,  
then \[\Nm(D_{n,a}) \ll_{\epsilon, K,a}  |\Nm(a^n -1)|^{\epsilon}.\]
\end{lemma} 
                                                                                                                                                                                                                                                                                                                                                                                                                                                                                                                                                                                                                                                                                                                                                                                                                                                                                                                                                                                                                                                                                                                                                                                                                                                                                                                                                                                                                                                                                                                                                                                                                                                                                                                                                                                                                                                                                                                                                                                                                                                                                                                                        
\begin{proof} Setting  $\alpha = a^n -1$ and $\beta = -a^n$ in Conjecture \ref{nf_abc_conj} implies

\begin{align*} |\Nm(a^n -1)| &
\ll_{\epsilon, K}
 \left (  \prod_{\p | (a)} \Nm(\p) \right )^{1+\epsilon /2}    \left (\prod_{\p | (a^n -1)} \Nm(\p)   \right )^{1 + \epsilon /2}.\\
\intertext{The powerful part of $(a^n-1)$ is $D_{n,a}$, so
$\text{rad}(D_{n,a}) \leq D_{n,a}^{1/2}$.  The substitution $(a^n -1) = C_{n,a} D_{n,a}$ yields}
\Nm(C_{n,a}) \Nm(D_{n,a}) 
& \ll_{\epsilon, K} 
 \left (  \prod_{\p | (a)} \Nm(\p) \right )^{1+\epsilon /2}  \Nm(C_{n,a})^{1 + \epsilon /2} \Nm(D_{n,a})^{\frac{1 + \epsilon /2}{2}},\\
 \intertext{so}
 \Nm(D_n(a))^{\frac{1}{2}} & \ll_{\epsilon, K}
 \left (  \prod_{\p | (a)} \Nm(\p) \right )^{1+\epsilon/2}  \Nm(C_{n,a} D_{n,a})^{\epsilon /2}\\
 \intertext{and}
 \Nm(D_n(a)) & \ll_{\epsilon, K}
 \left (   \prod_{\p | (a)} \Nm(\p) \right )^{2+\epsilon} |\Nm(a^n -1)|^{\epsilon}.
 \end{align*}
 We absorb $ \left (   \prod_{\p | (a)} \Nm(\p) \right )^{2+\epsilon}$ into the $\ll_{\epsilon, K}$ and use $\ll_{\epsilon, K,a}$,
 finishing the proof.
 \end{proof}
 Applying Lemma \ref{upper_norm_bound} gives the following corollary.
 \begin{coro}\label{second_bound_on_D} If $K$ is a Galois number field, if $\epsilon >0$, if $a \in \O_K \setminus \{0\}$ is not a root of unity, and if the $abc$ conjecture holds in $K$,  
then \[\Nm(D_{n,a}) \ll_{\epsilon, K,a}  |\Nm(a)|^{n\epsilon}.\]
\end{coro}


 \begin{lemma}\label{second_bound_on_C}  If $K$ is a Galois number field, the $abc$ conjecture holds in $K$, $a \in \O_K \setminus \{0\}$ is not a root of unity,
 and $\epsilon >0$, then 
    \[\Nm(C_{n,a}) \gg_{ \epsilon, K, a} |\Nm(a)|^{n(1 -\epsilon)}.\]
 \end{lemma}
 
\begin{proof} Since $(a^n - 1) + 1 = a^n$, the $abc$ conjecture for $K$ tells us 
\begin{align*}
|\Nm(a)|^n & \ll_{\epsilon, K}
\left (   
\prod_{\p | ( a ^n ) }  \Nm(\p)  \prod_{\p | ( a ^n -1 ) }  \Nm(\p) 
\right )^{1 + \epsilon /2},\\
\intertext{so}
|\Nm(a)|^{n} & \ll_{\epsilon, K, a}
\left (   \prod_{\p | C_{n,a}}  \Nm(\p)  \prod_{\p | D_{n,a}}  \Nm(\p) \right )^{1 + \epsilon /2}.\\
\intertext{Recalling $\text{rad}(D_{n,a}) = \prod_{\p | D_{n,a}}  \Nm(\p)  \leq \Nm(D_{n,a})^{1/2}$ lets us rewrite 
as}
|\Nm(a)|^{\frac{n}{1 + \epsilon/2}} & \ll_{\epsilon, K, a}
\Nm(C_{n,a}) \Nm(D_{n,a})^{1/2}.\\
\intertext{We apply Corollary \ref{second_bound_on_D} and use $\epsilon$, rather than $\epsilon /2$, to see}
 |\Nm(a)|^{\frac{n}{1 + \epsilon/2}} & \ll_{\epsilon, K, a}
 \Nm(C_{n,a}) |\Nm(a)|^{\frac{n\epsilon}{2}}.\\
\intertext{Since $\frac{1}{1 + \epsilon /2} - \epsilon/2 > 1 - \epsilon$, we conclude}
\Nm(C_{n,a}) & \gg_{\epsilon, K, a} |\Nm(a)|^{n(1 - \epsilon)}.
\end{align*}
\end{proof}

\begin{coro} \label{new_prime}
If $K$ is a Galois number field, if $a \in \O_K \setminus \{0\}$ is not a root of unity, if $k \in \Z^+$ is fixed, and if
  $q$ is a large rational prime,
then there exists a prime $\p$ such that $\p | C_{kq,a}$, but $\p \nmid C_{km,a}$ for any $m < q$.
\end{coro}
\begin{proof} (Adapted from Lemma 3.2, \cite{Graves}, p 1812)
Since $|\Nm(a)| \geq 2$ and $\Nm(C_{kn,a}) \gg_{\epsilon, K, a} |\Nm(a)|^{kn(1 - \epsilon)}$, 
Lemma \ref{upper_norm_bound} implies 
$\Nm(C_{kn,a}) > 1 + |\Nm(a^k-1)|$ when $n$ is sufficiently large. 
Let us fix $q$ and suppose, leading to contraction, that if $\p | C_{kq,a}$, then $\p | C_{km,a}$ for some $m <q$.
Then, as $C_{kq,a}$ is square-free, 
\begin{align*} C_{kq,a} &| \prod_{m =1}^{q-1} ( C_{kq,a} + C_{km,a}).\\
\intertext{
The ideals $C_{kq,a}$ and $C_{km,a}$ respectively contain $a^{kq} -1$ and $a^{km} -1$, so the ideal
$(a^{kq} -1 ) + ( a^{km} -1 )$ is contained in the ideal $ C_{kq,a} + C_{km,a}$.  
The integers $q$ and $m$ are relatively prime, so the
 greatest common divisor of the polynomials $x^{km} -1 $ and $x^{kq} -1$ is $x^{\gcd(km,kq)} - 1 = x^k-1$.
 Therefore the ideal}
 ( a^k-1 ) & \subset ( a^{kq} -1 ) + ( a^{km} -1 ) \subset (C_{kq,a} + C_{km,a}) \\
 \intertext{and}
  C_{kq,a} &| \prod_{m =1}^{q-1} ( a^k-1 ).
  \end{align*}
The ideal $C_{kq,a}$ is radical, so $C_{kq,a}$ divides $( a^k - 1 )$ and $\Nm(C_{kq,a}) \leq |\Nm(a^k-1)|$, a contradiction.
 We conclude there must be a prime $\p$ that divides $C_{kq,a}$ but does not divide any $C_{km,a}$ for $m < q$.
  \end{proof}

\section{First Growth Result on non-Wieferich places}\label{Section:first_results}


\begin{proof} (of Theorem \ref{basic})
Suppose $q$ is a large rational prime.
Applying Lemma \ref{upper_norm_bound} shows $\Nm(C_{kq,a})\leq | \Nm(a^{kq} -1)| \leq 2^{[K:\Q]}|\Nm(a)|^{kq}$.  
Corollary \ref{new_prime} states
there exists a prime ideal $\p$ such that 
$\p | C_{kq,a}$, but $\p \nmid C_{km, a}$ for all $m <q$.
Together, Lemmas \ref{heart} and \ref{p_equiv_1} show that each of these primes $\p$ is a non-Wieferich place base $a$ 
where $\Nm(\p) \equiv 1 \pmod{k}$.

This all adds up to  
\[ \left  | \left \{
\begin{array}{c}
\text{non-Wieferich places }  \\
\p \text{ base }a
\end{array} :
\begin{array}{c}
\Nm(\p) \leq 2^{[K:\Q]} |\Nm(a)|^{kn},\\
 \Nm(\p) \equiv 1 \pmod{k}
 \end{array}  \right  \} \right | \gg_{\epsilon, K, a,k} \frac{n}{\log n}.\]
 Setting $x = 2^{[K:\Q]} |\Nm(a)|^{kn}$ demonstrates that
 \[ \left  | \left \{
\begin{array}{c}
\text{non-Wieferich places }  \\
\p \text{ base }a
\end{array} :
\Nm(\p) \leq x, \Nm(\p) \equiv 1 \pmod{k} \right \} \right | \gg_{\epsilon, K, a,k} \frac{\log x}{\log \log x}.\]
 \end{proof}

\section{Norms of Cyclotomic Polynomials in Number Fields} 

The proof of Theorem \ref{basic} relies on growth results in Section \ref{sub_norm_factors} for $a^n -1$ and its factors. 
Section \ref{sub_norm_factors} and Section \ref{Section:first_results} apply to all $a \in \O_K \setminus \{0\}$ that are not roots of unity.

  
The lemma below generalizes a bound of Thangadurai and Vatwani (Theorem 5 in \cite{TV}).
It works for a smaller selection of $a$ ---
the set of $a \in \O_K$ such that $|\sigma(a)| \geq 2$ for all $\sigma \in \Gal(K/\Q)$.
As Theorem \ref{iqf}'s proof shows, this subset is not much (if any) smaller than $\O_K \setminus \{0\}$ when $K$ is imaginary quadratic.
The following lemma shows that most elements of Galois number fields satisfy the assumptions of the subsequent result, Lemma \ref{lower_phi_bound}. 
For that later lemma,we denote the Euler totient function by $\phi$ and  the M\"{o}bius function by $\mu$. 

\begin{lemma}
Let $K$ be a Galois number field of degree $n$ over $\Q$ with basis $\{x_1, \dots, x_n\}$, so that 
$K = \Q(x_1, \ldots, x_n)$ and $\O_K = \Z[x_1, \ldots, x_n]$.
Given any integer $n-1$ tuple $(z_1, \ldots, z_{n-1}) \in \Z^{n-1}$, there are only finitely many $z_n \in \Z$ such that 
$\left | \sum_{i=1}^n z_i \sigma (x_i) \right | <2$ for some $\sigma \in \Gal(K/\Q)$.
\end{lemma}
\begin{proof} Fix $\sigma \in \Gal(K/\Q)$.
By the triangle inequality,
\begin{align*}
\left | \sum_{i=1}^n z_i \sigma (x_i) \right | &\leq \left | \sum_{i=1}^{n-1} z_i \sigma (x_i) \right | + |z_n \sigma (x_n)|.\\
\intertext{If}
\left | \sum_{i=1}^{n-1} z_i \sigma (x_i) \right | + |z_n \sigma (x_n)| & <2,\\
\intertext{then} 
|z_n| &\leq \frac{2 - \left | \sum_{i=1}^{n-1} z_i \sigma (x_i) \right |}{|\sigma(x_n)|}.
\end{align*}
There are only finitely many values of $z_n$ satisfying this inequality, and there are only $n$ elements of $\Gal(K/\Q)$, so our claim holds.
\end{proof}

\begin{lemma}\label{lower_phi_bound}  
Suppose that $K$ is a Galois number field and that $a \in K$ satisfies $|\sigma(a)| \geq 2$ for all $\sigma \in \Gal(K/\Q)$.  Then for all $n\geq 2$,
\[| \Nm(a)| ^{\phi(n)} \leq 2^{[K:\Q]}| \Nm(\Phi_n(a) ) |.\]
\end{lemma}
\begin{proof} (Adapted from Theorem 5 in \cite{TV})
Standard properties of the M\"{o}bius and totient functions show that
\begin{align*} 
\left |\Nm(\Phi_n(a))\right | 
 &= \left | \Nm \left (\prod_{d|n} (a^d -1)^{\mu(n/d)} \right ) \right |\\
&= |\Nm(a)^{\sum_{d|n} d\mu(n/d)} | \left | \prod_{d|n} \Nm \left (1 - \frac{1}{a^d} \right ) ^{\mu(n/d)} \right |\\
&=|\Nm(a)|^{\phi(n)} \prod_{d|n} \left |\Nm \left (1 -\frac{1}{a^d} \right ) \right |^{\mu(n/d)}
\end{align*}
and that 
\[\log \left (\prod_{d|n} \left | \Nm \left (1 - \frac{1}{a^d} \right ) \right |^{\mu(n/d)} \right )= 
\sum_{d|n} \mu(n/d) \log \left |\Nm \left (1 - \frac{1}{a^d} \right ) \right |.\]

By definition, 
\begin{align*}
\log \left | \Nm \left ( 1 - \frac{1}{a^d} \right ) \right | 
& =  \log \left ( \prod_{\sigma \in \Gal(K/\Q)} \left | 1 - \frac{1}{\sigma(a)^d} \right | \right )\\
& = \sum_{\sigma \in \Gal(K/\Q)} \log \left | 1 - \frac{1}{\sigma(a)^d} \right |.
\end{align*}
We use the triangle inequality to bound this below by $\displaystyle \sum_{\sigma \in \Gal(K/\Q)} \log \left ( 1 - |\sigma(a)|^{-d} \right ),$
and see that 
\begin{align*}
\sum_{d|n} \mu(n/d) \log \left |\Nm \left (1 - \frac{1}{a^d} \right ) \right | 
& \geq \sum_{d|n} \mu(n/d) \sum_{\sigma \in \Gal(K/\Q)} \log \left ( 1 - |\sigma(a)|^{-d} \right )\\
& = \sum_{\sigma \in \Gal(K/\Q)}\sum_{d|n} \mu(n/d)\log \left ( 1 - |\sigma(a)|^{-d} \right ).
\end{align*}
In their proof, Thangadurai and Vatwani showed that if $n$ and $b$ are integers $\geq 2$, then 
\begin{equation}\label{sandwich}
- \log 2 \leq \sum_{d|n} \mu(n/d) \log (1 - b^{-d}) \leq \log 2.
\end{equation}
The details of their argument, however, never require that $b$ be an integer, so they actually showed that if $n$ is an integer $\geq 2$ and 
if $b$ is a real number $\geq 2$, then equation \ref{sandwich} holds.  This then implies  
\[\sum_{\sigma \in \Gal(K/\Q)}\sum_{d|n} \mu(n/d)\log \left ( 1 - |\sigma(a)|^{-d} \right ) \geq -[K:\Q]\log 2,\]
so 
\[ \left |\Nm(\Phi_n(a))\right |  \geq \frac{|\Nm(a)|^{\phi(n)}}{ 2^{[K:\Q]}}.\]                                     
\end{proof}

We use this to give a lower bound for $\Nm(C_{n,a}')$, which allows us to study primes in arithmetic progressions.

\begin{lemma} \label{prime bound} If $K$ is a  Galois number field, if the $abc$ conjecture holds in $K$,  if $a \in \O_K$, and if
 $\min_{\sigma \in \Gal(K/\Q)} |\sigma(a)| \geq 2$,  then 
  \[\Nm(C'_{n,a}) \gg_{\epsilon, K, a} |\Nm(a)|^{\phi(n) - n\epsilon}.\]
  \end{lemma}
  
\begin{proof}  Definition \ref{factorizations} factors $(\Phi_n(a)) = C_{n,a}' D_{n,a}'$, so by Lemma \ref{lower_phi_bound},
\begin{align*}
\Nm(C_{n,a}') \Nm(D_{n,a}') &= |\Nm(\Phi_n(a))| \geq |\Nm(a)|^{\phi(n)}/2^{[K:\Q]}.\\
\intertext{ We bound $\Nm(D_{n,a}')$ above by $\Nm(D_{n,a})$ and apply Corollary \ref{second_bound_on_D} to see}
 \Nm(C_{n,a}') |\Nm(a)|^{n \epsilon} &\gg_{\epsilon, K, a} |\Nm(a)|^{\phi(n)},
 \end{align*}
and derive $\Nm(C_{n,a}') \gg_{\epsilon, K, a}  |\Nm(a)|^{\phi(n) - n\epsilon} $.
\end{proof}


\section{A Stronger Growth Result on Weiferich Places}\label{Section:mainresults}

In order to generalize Ding's growth result to algebraic number fields, we need the next three lemmas.

\begin{lemma}  \label{order_exercise}
 If $K$ is a number field, if $a \in \O_K \setminus \{0\}$ is not a root of unity, if $\p$ is an unramified prime ideal of $\O_K$ lying over the rational
prime $p$, and if $p \nmid m$, then $\Phi_m(a) \in \p$ if and only if $e_{\p}(a) =m$.
\end{lemma}
\begin{proof}
We restate Murty's solution to Exercise 1.5.29 in \cite{murty_book}, adjusted to algebraic integers and prime ideals in $\O_K$.

If $e_{\p}(a) =m$, then $\p$ divides $( a^m -1 ) = \prod_{d|m} ( \Phi_d(a) )$
and one of the $\Phi_d(a)$ must be an element of $\p$.
If $\Phi_d(a) \in \p$, then $\p$ divides $( a^d -1 ) = \prod_{\delta|d} ( \Phi_{\delta}(a))  $, 
and the order of $a$ modulo $\p$ must divide $d$.  We just showed $d|m$ and $m|d$, so $d =m$ and $\Phi_m(a) \in \p$.

Now suppose $\Phi_m(a) \in \p$ and $p \nmid m$, so $a^m \equiv 1 \pmod{\p}$ and $e_{\p}(a) |m$.
Assume, leading to a contradiction, that $e_{\p}(a) = d <m$.
Since $p \in \p$,  
\begin{align*}\Phi_m(a) &\equiv \Phi_m(a +p) \equiv 0 \equiv \Phi_d(a) \equiv \Phi_d(a +p) \mod{\p},\\
\intertext{and therefore $\Phi_m(a)\Phi_d(a) \equiv \Phi_m(a+p) \Phi_d(a + p) \equiv 0 \mod{\p^2}$.
The product $\Phi_m(x) \Phi_d(x)$ divides $x^m -1$, so}
a^m -1 & \equiv (a+p)^m - 1 \equiv a^m + mpa^{m-1} -1 \pmod{\p^2}.
\end{align*}
and $mpa^{m-1} \in \p^2$. 
We assumed $\p$ is unramified, so $p$ is an element of $\p$, but it is not an element of $\p^2$.
This implies $ma^{m-1} \in \p$, and since $p$ does not divide the integer $m$,  the element $a^{m-1}$ is in $\p$.
Primality means $a \in \p$ and therefore $\p \nmid \Phi_{m}(a)$, a contradiction.  
The result follows.
\end{proof} 
We conclude that if $p \nmid m$ and $\Phi_m(a) \in \p$, then $m$ divides the group order, $\Nm(\p) -1$, and $\Nm(\p) \equiv 1 \pmod{m}$.

\begin{lemma}\label{carrot}
If $K$ is a Galois number field, if $a \in \O_K \setminus \{0\}$  is not a root of unity, if $\p$ is an unramified prime ideal of $\O_K$ lying over the rational
prime $p$, and if $\Phi_{n}(a) \in \p$, then the order of $a$ modulo $\p$ is $n \cdot p^{-v_p(n)}$ and $\Nm(\p) \equiv 1 \pmod{n \cdot p^{-v_p(n)}}.$
\end{lemma}
\begin{proof} Lemma \ref{order_exercise} is the case where $v_p(n) =0$, so let $n = p^r m$, with $r \geq 1$ and $p \nmid m$.
Then
\[\Phi_n(x) = \Phi_{pm}(x^{p^{r-1}}) = \frac{\Phi_m(x^{p^r})}{\Phi_m(x^{p^{r-1}})}, \]
and $\Phi_n(a) \in \p$ implies $\Phi_m(a^{\p^r}) \in \p$.
By assumption, $a^{p^r}$ is neither $0$ nor a root of unity, so $e_{\p}(a^{p^r}) = m$ by Lemma \ref{order_exercise}.
The rational integers $p^r$ and $\Nm(\p) -1$ are relatively prime, 
so $e_{\p}(a^{p^r}) = e_{\p}(a) = m = n \cdot p^{-v_p(n)}$.
This means $n \cdot p^{-v_p(n)}$ divides $\Nm(\p) -1$ and  $\Nm(\p) \equiv 1 \pmod{n \cdot p^{-v_p(n)}}$.
\end{proof}

Given a positive rational integer $k$, we define $c(k) = \prod_{p|k} (1 - \frac{\gcd(k, p)}{p^2}) >0$.
In the proof Lemma 2.6 in \cite{Ding}, Ding showed the following.
\begin{lemma}\label{Ding} (\cite{Ding}, p 485)
Given a positive integer $k$, 
\[ \left | \left \{ n \leq x: \phi(nk) > \frac{2 c(k)}{3} nk \right \} \right | \gg_k x.\]
\end{lemma}

The previous three lemmas allow us to generalize Ding's growth result as Theorem \ref{strong}, 
but only for algebraic integers $a$ such that $|\sigma(a)| \geq 2$ for all $\sigma \in \Gal(K/\Q)$.
Theorem \ref{iqf} states that is not a problem for imaginary quadratic number fields.
  
\begin{proof} (of Theorem \ref{strong}) Lemma \ref{Ding} shows 
\begin{align*}
\left | \left \{ n \leq x: \phi(nk) > \frac{2 c(k)}{3} nk \right \} \right | & \gg_k x.\\
\intertext{If $\phi(nk) > \frac{2 c(k)}{3} nk$ and $\epsilon = c(k)/3$, then Lemma \ref{prime bound} implies}
\Nm(C_{nk,a}') & \gg_{K, a, k} |\Nm(a)|^{\frac{c(k) n k}{3}}.\\
\intertext{That means that if $n$ is large enough and $\phi(nk) > \frac{2c(k)nk}{3}$, 
$\Nm(C_{nk,a}') >1$, so }
\left | \left \{ n \leq x : \Nm(C_{nk, a}') > 1 \right \} \right | &\gg_{K, a, k}  x.
 \stepcounter{equation}\tag{\theequation}\label{carrot}
\end{align*}

 Lemma \ref{Chen_lemma} shows that if $\Nm(C_{nk,a}') >1$,  there exists some prime $\p$ such that 
$\p|C_{nk,a}'$ but $\p \nmid C_{m,a}'$ for $m < nk$.
Lemma \ref{heart} shows that each of these primes$\p$ is a non-Wieferich place base $a$.
Suppose $\p$ lies over the rational prime $p$.
Lemma \ref{carrot} shows that if $\p$ is unramified and $\gcd(p,k)=1$ then $\Nm(\p) \equiv 1 \pmod{k}$.
There are only finitely many ramified primes and only finitely many primes $p$ divide $k$, so this is not an onerous condition.
We use Lemma \ref{upper_norm_bound} again to see that if $\p | C_{nk, a}'$, then
\[\Nm(\p)  \leq \Nm(C_{nk,a}') \leq |\Nm(a^{nk}-1)| \leq 2^{[K:\Q]} |\Nm(a)|^{nk},\]
so applying our analysis to equation \ref{carrot} shows 
\begin{align*}
\left | \left \{
\begin{array}{l}
\text{non-Wieferich places }\\
 \p \text{ base }a
 \end{array} : 
 \begin{array}{l}
 \Nm(\p) \leq 2^{[K:\Q]} |\Nm(a)|^{yk} \\
 \Nm(\p) \equiv 1 \pmod{k} 
 \end{array} 
\right \} \right | & \gg_{K, a, k} y,\\
\intertext{and thus setting $x = 2^{[K:\Q]} |\Nm(a)|^{yk}$ yields}
\left | \left \{
\begin{array}{l}
\text{non-Wieferich places }\\
 \p \text{ base }a
 \end{array} : 
 \begin{array}{l}
 \Nm(\p) \leq x \\
 \Nm(\p) \equiv 1 \pmod{k} 
 \end{array} 
\right \} \right | & \gg_{K, a, k} \log x.\\
\end{align*} 
\end{proof}

\begin{proof} (of Theorem \ref{iqf}) Theorem \ref{strong} shows our claim holds if $|\sigma(a)| \geq 2$ for all $\sigma \in \Gal(K/\Q)$.  
Suppose $K = \Q(\sqrt{-d})$.

If 
$d \not \equiv 3 \pmod{4}$ and $a = x + y \sqrt{-d}$,  
then our condition holds unless $|x + y \sqrt{-d}| <2$.  
That happens if and only if $x^2 + d y^2 <4$, which would force $a \in \{0, \pm 1, \pm i, \pm 1 \pm i, \pm \sqrt{-2}, \pm 1 \pm \sqrt{-2}\}$.

If 
$d \equiv 3 \pmod{4}$ and $a = x + y \frac{1 + \sqrt{-d}}{2}$,  
then our condition holds unless $|x + y \frac{1 + \sqrt{-d}}{2}| <2$. 
This happens if and only if $x^2 + xy + \frac{(d +1)y^2}{4} <4$,
which only occurs if $x \in \left \{ 0, \pm 1, \frac{\pm 1 \pm \sqrt{-3}}{2},  \frac{\pm 3 \pm \sqrt{-3}}{2},
\frac{\pm 1 \pm \sqrt{-7}}{2},\frac{\pm 1 \pm \sqrt{-11}}{2} \right \}$.
\end{proof}

\bmhead{Acknowledgements}  
The first author would like to thank Professor Murty for introducing her to Wieferich and 
non-Wieferich primes. We would also like to thank Dubi Kelmer, Greg Minton, and Michael Robinson for their help.  

\section*{Declarations}
The authors received support from IDA's Center for Computing Sciences and the University of Maine.

\end{document}